\documentclass[12pt]{article}
\usepackage{amsmath}
\usepackage{amsthm}
\usepackage{amsfonts}
\usepackage{amssymb}
\usepackage{esint}
\usepackage{eucal}
\usepackage{mathrsfs}
\usepackage{graphicx,graphics}
\numberwithin{equation}{section}
\usepackage{graphicx,graphics}
\usepackage{pdfsync}
\usepackage{multirow}
\usepackage{endnotes}
\def \dis {\displaystyle}
\def \tex {\textstyle}

\def \confai {-\kern -.5em\rightharpoonup}
\def \cqfd {\hfill$\Box$}
\def \div{\mbox{\rm div}}
\def \Div{\mbox{\rm Div}}
\def \Curl{\mbox{\rm Curl}}

\def \al {\alpha}
\def \be {\beta}
\def \ga {\gamma}

\def \De {\Delta}
\def \ep {\varepsilon}

\def \Om {\Omega}
\def \la {\lambda}

\def \ph {\varphi}

\def \ka {\kappa}
\def \si {\sigma}
\def \Si {\Sigma}

\def \ZZ {\mathbb Z}

\def \RR {\mathbb R}

\def \D {\mathscr{D}}

\def \beq {\begin{equation}}
\def \eeq {\end{equation}}
\def \ba {\begin{array}}
\def \ea {\end{array}}
\def \bs {\bigskip}
\def \ms {\medskip}
\def \ss {\smallskip}
\def \ecart {\noalign{\medskip}}
\newtheorem{Thm}{Theorem}[section]

\newtheorem{Pro}[Thm]{Proposition}
\newtheorem{Lem}[Thm]{Lemma}

\newtheorem{Adef}[Thm]{Definition}
\newenvironment{Def}{\begin{Adef}\rm}{\end{Adef}}
\newtheorem{Arem}[Thm]{Remark}
\newenvironment{Rem}{\begin{Arem}\rm}{\end{Arem}}
\newtheorem{Aexa}[Thm]{Example}
\newenvironment{Exa}{\begin{Aexa}\rm}{\end{Aexa}}
\newtheorem{Anot}[Thm]{Notation}

\def \refe #1.{(\ref{#1})}
\def \reff #1.{figure~\ref{#1}}
\def \refs #1.{Section~\ref{#1}}
\def \refss #1.{Subsection~\ref{#1}}
\def \refD #1.{Definition~\ref{#1}}
\def \refT #1.{Theorem~\ref{#1}}
\def \refL #1.{Lemma~\ref{#1}}
\def \refC #1.{Corollary~\ref{#1}}
\def \refP #1.{Proposition~\ref{#1}}
\def \refPt #1.{Properties~\ref{#1}}
\def \refR #1.{Remark~\ref{#1}}
\def \refE #1.{Example~\ref{#1}}
\def \refN #1.{Notation~\ref{#1}}
\newcounter{marnote}

\title{Which electric fields are realizable in conducting materials?}
\begin{document}
\maketitle
\vskip -.5cm
\centerline{\large
Marc Briane\footnote{Institut de Recherche Math\'ematique de Rennes, INSA de Rennes, FRANCE -- mbriane@insa-rennes.fr,},
\quad
Graeme W. Milton\footnote{Department of Mathematics, University of Utah, USA -- milton@math.utah.edu,},
\quad
Andrejs Treibergs\footnote{Department of Mathematics, University of Utah, USA -- treiberg@math.utah.edu.}
}
\vskip 1.cm
\begin{abstract}
In this paper we study the realizability of a given smooth periodic gradient field $\nabla u$ defined in $\RR^d$, in the sense of finding when one can obtain a matrix conductivity $\si$ such that $\si\nabla u$ is a divergence free current field. The construction is shown to be always possible locally in $\RR^d$ provided that $\nabla u$ is non-vanishing. This condition is also necessary in dimension two but not in dimension three. In fact the realizability may fail for non-regular gradient fields, and in general the conductivity cannot be both periodic and isotropic. However, using a dynamical systems approach the isotropic realizability is proved to hold in the whole space (without periodicity) under the assumption that the gradient does not vanish anywhere. Moreover, a sharp condition is obtained to ensure the isotropic realizability in the torus. The realizability of a matrix field is also investigated both in the periodic case and in the laminate case. In this context the sign of the matrix field determinant plays an essential role according to the space dimension.
\end{abstract}
\vskip .5cm\noindent
{\bf Keywords~:} Conductivity, Electric field, Dynamical systems
\par\bs\noindent
{\bf Mathematics Subject Classification~:} 35B27, 78A30, 37C10
\section{Introduction}
The mathematical study of composite media has grown remarkably since the seventies through the asymptotic analysis of pde's governing their behavior (see, {\em e.g.}, \cite{BLP}, \cite{BaPa}, \cite{JKO}, \cite{Mil2}). In the periodic framework of the conductivity equation, the derivation of the effective (or homogenized) properties of a given composite conductor in $\RR^d$, with a periodic matrix-valued conductivity $\si$, reduces to the cell problem of finding periodic gradients $\nabla u$ solving
\beq\label{cellpb}
\div\left(\si\nabla u\right)=0\quad\mbox{in }\RR^d,
\eeq
which gives the effective conductivity $\si^*$ {\em via} the average formula
\beq\label{effcon}
\si^*\langle\nabla u\rangle=\langle\si\nabla u\rangle.
\eeq
Note that the periodicity condition is not actually a restriction, since by~\cite{Rai} (see also \cite{All}, Theorem~1.3.23) any effective matrix can be shown to be a pointwise limit of a sequence of periodic homogenized matrices. In equation \refe{cellpb}. the vector-valued function $\nabla u$ represents the electric field, while $\si\nabla u$ is the current field according to Ohm's law. Alternatively we can consider a vector-valued potential $U$
with gradient $DU$ where each component of $U$ satisfies~\refe{cellpb}.. In this case the components of $U$ represent the potentials obtained for different applied fields, and $DU$ will be referred to as the matrix-valued electric field. Going back to the original conductivity problem it is then natural to characterize mathematically among all periodic gradient fields those solving the conductivity equation \refe{cellpb}. for some positive definite symmetric periodic matrix-valued function $\si$. In other words the question is to know which electric fields are realizable. On the other hand, this work is partly motivated by the search for sharp bounds on the effective moduli of composites. This search has led investigators to derive as much information as possible about fields in composites. A prime example is given by the positivity of the determinant of periodic matrix-valued electric fields in two dimensions obtained by Alessandrini and Nesi \cite{AlNe}. This led to sharp bounds on effective moduli for three phase conducting composites (see, {\em e.g.}, \cite{Nes,ChZh}). Therefore, a natural question to ask, which we address here, is: what are the conditions on a gradient to be realizable as an electric field?
\par
In Section~\ref{s.vec} we focus on vector-valued electric fields.
First of all, due to the rectification theorem we prove (see Theorem~\ref{thm.v1}) that any non-vanishing smooth gradient field $\nabla u$ is isotropically realizable locally in $\RR^d$, in the sense that in the neighborhood of each point equation \refe{cellpb}. holds for some isotropic conductivity $\si I_d$. Two examples show that the regularity of the gradient field is essential, and that the periodicity of $\si$ is not satisfied in general.
Conversely, in dimension two the realizability of a smooth periodic gradient field $\nabla u$ implies that $\nabla u$ does not vanish in~$\RR^2$. This is not the case in dimension three as exemplified by the periodic chain-mail of~\cite{BMN}. Again in dimension two a necessary and sufficient condition for the (at least anisotropic) realizability is given (see Theorem~\ref{thm.v2}).
Then, the question of the global isotropic realizability is investigated through a dynamical systems approach. On the one hand, considering the trajectories along the gradient field $\nabla u$ which cross a fixed hyperplane, we build (see Proposition~\ref{pro.isoreaH}) an admissible isotropic conductivity $\si$ in the whole space. The construction is illustrated with the potential $u(x):=x_1-\cos(2\pi x_2)$ in dimension two. On the other hand, upon replacing the hyperplane by the equipotential $\{u=0\}$, a general formula for the isotropic conductivity $\si$ is derived (see Theorem~\ref{thm.isoreaRd}) for any smooth gradient field in $\RR^d$. Finally, a sharp condition for the isotropic realizability in the torus is obtained (see Theorem~\ref{thm.isoreaY}), which allows us to construct a periodic conductivity~$\si$.
\par
Section~\ref{s.mat} is devoted to matrix-valued fields. The goal is to characterize those smooth potentials $U=(u_1,\dots,u_d)$ the gradient $DU$ of which is a realizable periodic matrix-valued electric field. When the determinant of $DU$ has a constant sign, it is proved to be realizable with an anisotropic matrix-valued conductivity $\si$. This can be achieved in an infinite number of ways using Piola's identity coming from mechanics (see Theorem~\ref{thm.m1} and Proposition~\ref{pro.siJDU}).
This yields a necessary and sufficient realizability condition in dimension two due to the determinant positivity result of \cite{AlNe}. However, the periodic chain-mail example of \cite{BMN} shows that this condition is not necessary in dimension three. We extend (see Theorem~\ref{thm.m2lam}) the realizability result to (non-regular) laminate matrix fields having the remarkable property of a constant sign determinant in any dimension (see \cite{BMN}, Theorem~3.3).
\newpage
\subsubsection*{Notations}
\begin{itemize}
\item $\left(e_1,\dots,e_d\right)$ denotes the canonical basis of $\RR^d$.
\item $I_d$ denotes the unit matrix of $\RR^{d\times d}$, and $R_\perp$ denotes the $90^\circ$ rotation matrix in $\RR^{2\times 2}$.
\item For $A\in\RR^{d\times d}$, $A^T$ denotes the transpose of the matrix $A$.
\item For $\xi,\eta\in\RR^d$, $\xi\otimes\eta$ denotes the matrix $\left[\xi_i\,\eta_j\right]_{1\leq i,j\leq d}$.
\item $Y$ denotes any closed parallelepiped of $\RR^d$, and $Y_d:=\left[0,1\right]^d$.
\item $\langle\cdot\rangle$ denotes the average over $Y$.
\item $C^k_\sharp(Y)$ denotes the space of $k$-continuously differentiable $Y$-periodic functions on $\RR^d$.
\item $L^2_\sharp(Y)$ denotes the space of $Y$-periodic functions in $L^2_{\rm }(\RR^d)$, and $H^1_\sharp(Y)$ denotes the space of functions $\ph\in L^2_\sharp(Y)$ such that $\nabla\ph\in L^2_\sharp(Y)^d$.
\item For any open set $\Om$ of $\RR^d$, $C^\infty_c(\Om)$ denotes the space of  smooth functions with compact support in $\Om$, and $\D'(\Om)$ the space of distributions on $\Om$.
\item For $u\in C^1(\RR^d)$ and $U=(U_j)_{1\leq j\leq d}\in C^1(\RR^d)^d$,
\beq
\nabla u:=\left({\partial u\over\partial x_i}\right)_{1\leq i\leq d}\quad\mbox{and}\quad
DU:=\big(\nabla U_1,\dots,\nabla U_d\big)=\left[{\partial U_j\over\partial x_i}\right]_{1\leq i,j\leq d}.
\eeq
The partial derivative $\dis {\partial u\over\partial x_i}$ will be sometimes denoted $\partial_i u$.
\item For $\Si=\left[\Si_{ij}\right]_{1\leq i,j\leq d}\in C^1(\RR^d)^{d\times d}$,
\beq
\Div\left(\Si\right):=\left(\sum_{i=1}^d{\partial \Si_{ij}\over\partial x_i}\right)_{1\leq j\leq d}\quad\mbox{and}\quad
\Curl\left(\Si\right):=\left({\partial \Si_{ik}\over\partial x_j}-{\partial \Si_{jk}\over\partial x_i}\right)_{1\leq i,j,k\leq d}.
\eeq
\item For $\xi^1_1,\dots,\xi^{d-1}$ in $\RR^d$, the cross product $\xi^1\times\cdots\times \xi^{d-1}$ is defined by
\beq
\xi\cdot\left(\xi^1\times\cdots\times \xi^{d-1}\right)=\det\left(\xi,\xi^1,\dots,\xi^{d-1}\right),\quad\mbox{for any }\xi\in\RR^d,
\eeq
where $\det$ is the determinant with respect to the canonical basis $(e_1,\dots,e_d)$, or equivalently, the $k^{\rm th}$ coordinate of the cross product is given by 
\beq\label{crossprod}
\left(\xi^1\times\cdots\times \xi^{d-1}\right)\cdot e_k=(-1)^{k+1}\left|\,\begin{smallmatrix}
\xi^1_1 & \cdots & \xi^{d-1}_1
\\
\vdots & \qquad\ddots\qquad & \vdots
\\
\\
\xi^1_{k-1} & \cdots & \xi^{d-1}_{k-1}
\\
\xi^1_{k+1} & \cdots & \xi^{d-1}_{k+1}
\\
\vdots & \qquad\ddots\qquad & \vdots
\\
\\
\xi^1_d & \cdots & \xi^{d-1}_d
\end{smallmatrix}\,\right|.
\eeq
\end{itemize}
\section{The vector field case}\label{s.vec}
\begin{Def}\label{def.vfield}
Let $\Om$ be an (bounded or not) open set of $\RR^d$, $d\geq 2$, and let $u\in H^1_{\rm}(\Om)$. The vector-valued field $\nabla u$ is said to be a {\em realizable} electric field in $\Om$ if there exist a symmetric positive definite matrix-valued $\si\in L^\infty_{\rm loc}(\Om)^{d\times d}$
such that
\beq
\div\left(\si\nabla u\right)=0\quad\mbox{in }\D'(\Om).
\eeq
If $\si$ can be chosen isotropic ($\si\to\si I_d$), the field $\nabla u$ is said to be {\em isotropically realizable} in $\Om$.
\end{Def}
\subsection{Isotropic and anisotropic realizability}
\subsubsection{Characterization of an isotropically realizable electric field}
\begin{Thm}\label{thm.v1}
Let $Y$ be a closed parallelepiped of $\RR^d$. Consider $u\in C^1(\RR^d)$, $d\geq 2$, such that
\beq\label{graduper}
\nabla u\mbox{ is $Y$-periodic}\quad\mbox{and}\quad\langle\nabla u\rangle\neq 0.
\eeq
\begin{itemize}
\item[$i)$] Assume that
\beq\label{gradu­0}
\nabla u\neq 0\quad\mbox{everywhere in }\RR^d.
\eeq
Then, $\nabla u$ is an isotropically realizable electric field locally in $\RR^d$ associated with a continuous conductivity.
\item[$ii)$] Assume that $\nabla u$ satisfies condition \refe{graduper}., and is a realizable electric field in $\RR^2$ associated with a smooth $Y$-periodic conductivity. Then, condition \refe{gradu­0}. holds true.
\item[$iii)$] There exists a gradient field $\nabla u$ satisfying \refe{graduper}., which is a realizable electric field in $\RR^3$ associated with a smooth $Y_3$-periodic conductivity, and which admits a critical point $y_0$, {\em i.e.} $\nabla u(y_0)=0$.
\end{itemize}
\end{Thm}
\begin{Rem}\label{rem.thm1}
Part $i)$ of Theorem~\ref{thm.v1} provides a local result in the smooth case, and still holds without the periodicity assumption on $\nabla u$. It is then natural to ask if the local result remains valid when the potential $u$ is only Lipschitz continuous. The answer is negative as shown in Example~\ref{exa1} below.
We may also ask if a global realization of a periodic gradient can always be obtained with a periodic isotropic conductivity $\sigma$. 
The answer is still negative as shown in Example~\ref{exa2}.
\par
The underlying reason for these negative results is that the proof of Theorem~\ref{thm.v1} is based on the rectification theorem which needs at least $C^1$-regularity and is local.
\end{Rem}
\begin{Exa}\label{exa1}
Let $\chi:\RR\to\RR$ be the $1$-periodic characteristic function which agrees with the characteristic function of $[0,1/2]$ on $[0,1]$. Consider the function $u$ defined in $\RR^2$ by
\beq\label{uexa1}
u(x) := x_2-x_1 + \int_0^{x_1}\chi(t)\,dt,\quad\mbox{for any }x=(x_1,x_2)\in\RR^2.
\eeq
The function $u$ is Lipschitz continuous, and
\beq\label{graduexa1}
\nabla u=\chi\,e_2+\left(1-\chi\right)\left(e_2-e_1\right)\quad\mbox{a.e. in }\in\RR^2.
\eeq
The discontinuity points of $\nabla u$ lie on the lines $\{x_1=1/2\left(1+k\right)\}$, $k\in\ZZ$.
Let $Q:=(-r,r)^2$ for some $r\in(0,1/2)$.
\par
Assume that there exists a positive function $\si\in L^\infty(Q)$ such that $\si\nabla u$ is divergence free in $Q$. Let $v$ be a stream function such that $\si\nabla u=R_\perp\nabla v$ a.e. in $Q$. The function $v$ is unique up to an additive constant, and is Lipschitz continuous.
On the one hand, we have
\beq
0=\nabla u\cdot\nabla v=(e_2-e_1)\cdot\nabla v\quad\mbox{a.e. in }(-r,0)\times(-r,r),
\eeq
hence $v(x)=f(x_1+x_2)$ for some Lipschitz continuous function $f$ defined in $[-2r,r]$. On the other hand, we have
\beq
0=\nabla u\cdot\nabla v=e_2\cdot\nabla v\quad\mbox{a.e. in }(0,r)\times(-r,r),
\eeq
hence $v(x)=g(x_1)$ for some Lipschitz continuous function $g$ in $[0,r]$. By the continuity of $v$ on the line $\{x_1=0\}$, we get that $f(x_2)=g(0)$, hence $f$ is constant in $[-r,r]$. Therefore, we have
\beq
\nabla v=0\;\;\mbox{a.e. in }(-r,0)\times(0,r)\quad\mbox{and}\quad\si\nabla u=\si\left(e_2-e_1\right)\neq 0\;\;\mbox{a.e. in }(-r,0)\times(0,r),
\eeq
which contradicts the equality $\si\nabla u=R_\perp\nabla v$ a.e. in $Q$. Therefore, the field $\nabla u$ is non-zero a.e. in $\RR^2$, but is not an isotropically realizable electric field in the neighborhood of any point of the lines $\{x_1=1/2\left(1+k\right)\}$, $k\in\ZZ$.
\end{Exa}
\begin{Rem}\label{rem.exa2}
The singularity of $\nabla u$ in Example~\ref{exa1} induces a jump of the current at the interface $\{x_1=0\}$. To compensate this jump we need to introduce formally an additional current concentrated on this line, which would imply an infinite conductivity there. The assumption of bounded conductivity (in $L^\infty$) leads to the former contradiction. Alternatively, with a smooth approximation of $\nabla u$ around the line $\{x_1=0\}$, then part~$i)$ of Theorem~\ref{thm.v1} applies which allows us to construct a suitable conductivity. But this conductivity blows up as the smooth gradient tends to $\nabla u$.
\end{Rem}
\begin{Exa}\label{exa2}
Consider the function $u$ defined in $\RR$ by
\beq\label{uexa2}
u(x):=x_1-\cos\left(2\pi x_2\right),\quad\mbox{for any }x=(x_1,x_2)\in\RR^2.
\eeq
The function $u$ is smooth, and its gradient $\nabla u$ is $Y_2$-periodic, independent of the variable $x_1$ and non-zero on $\RR^2$.
\par
Assume that there exists a smooth positive function $\si$ defined in $\RR^2$, which is $a$-periodic with respect to $x_1$ for some $a>0$, and such that $\si\nabla u$ is divergence free in $\RR^2$. Set $Q:=(0,a)\times (-r,r)$ for some $r\in(0,{1\over 2})$. By an integration by parts and taking into account the periodicity of $\si\nabla u$ with respect to $x_1$, we get that
\beq
\ba{ll}
\dis 0 & \dis =\int_Q\div\left(\si\nabla u\right)dx
\\ \ecart
& \dis =\underbrace{\int_{-r}^r\big(\si\nabla u(a,x_2)-\si\nabla u(0,x_2)\big)\cdot e_1\,dx_2}_{=0}
+\int_0^a\big(\si\nabla u(x_1,r)-\si\nabla u(x_1,-r)\big)\cdot e_2\,dx_1
\\
& \dis =2\pi\sin\left(2\pi r\right)\int_0^a\big(\si(x_1,r)+\si(x_1,-r)\big)\,dx_1>0,
\ea
\eeq
which yields a contradiction. Therefore, the $Y_2$-periodic field $\nabla u$ is not an isotropically realizable electric field in the torus.
\end{Exa}
\noindent
{\bf Proof of Theorem~\ref{thm.v1}.}
\par\ss\noindent
$i)$ Let $x_0\in\RR^d$. First assume that $d>2$. By the rectification theorem (see, {\em e.g.}, \cite{Arn}) there exist an open neighborhood $V_0$ of $x_0$, an open set $W_0$, and a $C^1$-diffeomorphism $\Phi:V_0\to W_0$ such that $D\Phi^T\,\nabla u=e_1$. Define $v_i:=\Phi_{i+1}$ for $i\in\{1,\dots,d-1\}$. Then, we get that $\nabla v_i\cdot\nabla u=0$ in~$V_0$, and the rank of $(\nabla v_1,\dots,\nabla v_{d-1})$ is equal to $(d-1)$ in $V_0$. Consider the continuous function
\beq\label{sigradvk}
\si:={|\nabla v_1\times\cdots\times\nabla v_{d-1}|\over|\nabla u|}>0\quad\mbox{in }V_0.
\eeq Since by definition, the cross product $\nabla v_1\times\cdots\times\nabla v_{d-1}$ is orthogonal to each $\nabla v_i$ as is $\nabla u$, then due to the condition~\refe{gradu­0}. combined with a continuity argument, there exists a fixed $\tau_0\in\{\pm 1\}$ such that
\beq
\nabla v_1\times\cdots\times\nabla v_{d-1}=\tau_0\,\si\nabla u\quad\mbox{in }V_0.
\eeq
Moreover, Theorem~3.2 of \cite{Dac} implies that $\nabla v_1\times\cdots\times\nabla v_{d-1}$ is divergence free, and so is $\si\nabla u$. Therefore, $\nabla u$ is an isotropically realizable electric field in $V_0$.
\par
When $d=2$, the equality $\nabla v_1\cdot\nabla u=0$ in $V_0$ yields for some fixed $\tau_0\in\{\pm 1\}$,
\beq\label{v1}
\tau_0\,R_\perp\nabla v_1=\underbrace{{|\nabla v_1|\over|\nabla u|}}_{\tex\si:=}\nabla u\quad\mbox{in }V_0,
\eeq
which also allows us to conclude the proof of $(i)$.
\par\ms\noindent
$ii)$ It is a straightforward consequence of \cite{AlNe} (Proposition~2, the smooth case).
\par\ms\noindent
$iii)$ Ancona \cite{Anc} first built an example of potential with critical points in dimension $d\geq 3$. The following construction is a regularization of the simpler example of \cite{BMN} which allows us to derive a change of sign for the determinant of the matrix electric field. Consider the periodic chain-mail $Q_\sharp\subset\RR^3$ of~\cite{BMN}, and the associated isotropic two-phase conductivity $\si^\ka$ which is equal to $\ka\gg 1$ in $Q_\sharp$ and to $1$ elsewhere. Now, let us modify slightly the conductivity $\si^\ka$ by considering a smooth $Y_3$-periodic isotropic conductivity $\tilde{\si}^\ka\in[1,\ka]$ which agrees with $\si^\ka$, except within a thin boundary layer of each interlocking ring $Q\subset Q_\sharp$, of width $\ka^{-1}$ from the boundary of~$Q$. Proceeding as in \cite{BMN} it is easy to prove that the smooth periodic matrix-valued electric field $D\tilde{U}^\ka$ solution of
\beq
\Div\,\big(\tilde{\si}^\ka D\tilde{U}^\ka\big)=0\;\;\mbox{in }\RR^3,\quad\mbox{with}\quad\langle D\tilde{U}^\ka\rangle=I_3,
\eeq
converges (as $\ka\to\infty$) strongly in $L^2(Y_3)^{3\times 3}$ to the same limit $DU$ as the electric field $DU^\ka$ associated with $\si^\ka$. Then, by virtue of \cite{BMN} $\det\left(DU\right)$ is negative around some point between two interlocking rings, so is $\det\big(D\tilde{U}^\ka\big)$ for $\ka$ large enough. This combined with $\big\langle\det\big(D\tilde{U}^\ka\big)\big\rangle=1$ and the continuity of $D\tilde{U}^\ka$, implies that there exists some point $y_0\in Y_3$ such that $\det\big(D\tilde{U}^\ka(y_0)\big)=0$. Therefore, there exists $\xi\in\RR^3\setminus\{0\}$ such that the potential $u:=\tilde{U}^\ka\cdot\xi$ satisfies $\langle\nabla u\rangle=\xi$ and $\nabla u(y_0)=D\tilde{U}^\ka(y_0)\,\xi=0$. Theorem~\ref{thm.v1} is thus proved. \cqfd
\subsubsection{Characterization of the anisotropic realizability in dimension two}
In dimension two we have the following characterization of realizable electric vector fields:
\begin{Thm}\label{thm.v2}
Let $Y$ be a closed parallelogram of $\RR^2$.
Consider a function $u\in C^1(\RR^2)$ satisfying~\refe{graduper}.. Then, a necessary and sufficient condition for $\nabla u$ to be a realizable electric field associated with a symmetric positive definite matrix-valued conductivity in $C^0_\sharp(Y)^{d\times d}$, is that there exists a function $v\in C^1(\RR^2)$ satisfying~\refe{graduper}. such that
\beq\label{gradugradv}
R_\perp\nabla u\cdot\nabla v=\det\left(\nabla u,\nabla v\right)>0\quad\mbox{everywhere in }\RR^2.
\eeq
\end{Thm}
\begin{Rem}\label{rem.v2lreg} 
The result of Theorem~\ref{thm.v2} still holds under the less regular assumption
\beq\label{graduL2per}
\nabla u\in L^2_\sharp(Y)^2,\quad\nabla u\neq 0\;\;\mbox{everywhere in }\RR^2\quad\mbox{and}\quad\langle \nabla u\rangle\neq 0.
\eeq
Then, the $Y$-periodic conductivity $\si$ defined by the formula \refe{sigraduv}. below is only defined almost everywhere in $\RR^2$, and is not necessarily uniformly bounded from below or above in the cell period $Y$. However, $\si\nabla u$ remains divergence 
free in the sense of distributions on $\RR^2$.
\end{Rem}
\noindent
{\bf Proof of Theorem~\ref{thm.v2}.}
\par\ss\noindent
{\it Sufficient condition:}
Let $u,v\in C^1(\RR^2)$ be two functions satisfying \refe{graduper}. and \refe{gradugradv}.. From~\refe{gradugradv}. we easily deduce that $\nabla u$ does not vanish in $\RR^2$. Then, we may define in $\RR^2$ the function
\beq\label{sigraduv}
\dis \si:={1\over|\nabla u|^4}\begin{pmatrix}\partial_1 u & \partial_2 u \\ -\partial_2 u & \partial_1 u\end{pmatrix}^T
\begin{pmatrix}R_\perp\nabla u\cdot\nabla v & -\nabla u\cdot\nabla v \\ -\nabla u\cdot\nabla v
& {|\nabla u\cdot\nabla v|^2+1\over R_\perp\nabla u\cdot\nabla v}\end{pmatrix}
\begin{pmatrix}\partial_1 u & \partial_2 u \\ -\partial_2 u & \partial_1 u\end{pmatrix}.
\eeq
Hence, $\si$ is a symmetric positive definite matrix-valued function in $C^0_\sharp(Y)^{d\times d}$ with determinant~$|\nabla u|^{-4}$. Moreover, a simple computation shows that $\si\nabla u=-\,R_\perp\nabla v$, so that $\si\nabla u$ is divergence free in~$\RR^d$. Therefore, $\nabla u$ is a realizable electric field in $\RR^d$ associated with the anisotropic conductivity $\si$.
\par\ms\noindent
{\it Necessary condition:} Let $u\in C^1(\RR)$ satisfying \refe{graduper}. such that $\nabla u$ is a realizable electric field associated with a symmetric positive definite matrix-valued conductivity $\si\in C^0_\sharp(Y)^{d\times d}$ in~$\RR^d$. Consider the unique (up to an additive constant) potential $v$ which solves $\div\left(\si\nabla v\right)=0$ in $\RR^d$, with $\nabla v\in H^1_\sharp(Y)^d$ and $\langle\nabla v\rangle=R_\perp\,\langle\nabla u\rangle$, and set $U:=(u,v)$.
By \refe{graduper}. we have
\beq
\det\big(\langle DU\rangle\big)=R_\perp\langle\nabla u\rangle\cdot\langle\nabla v\rangle=\big|\langle\nabla u\rangle\big|^2>0.
\eeq
Hence, due to \cite{AlNe} (Theorem~1) we have $\det\left(DU\right)>0$ a.e. in $\RR^2$. On the other hand, assume that there exists a point $y_0\in\RR^2$ such that $\det\left(DU\right)(y_0)=0$. Then, there exists $\xi\in\RR^2\setminus\{0\}$ such that the potential $u:=U\xi$ satisfies $\nabla u(y_0)=DU(y_0)\,\xi=0$, which contradicts Proposition~2 of \cite{AlNe} (the smooth case). Therefore, we get that $R_\perp\nabla u\cdot\nabla v=\det\left(DU\right)>0$ everywhere in~$\RR^2$, that is \refe{gradugradv}..
\cqfd
\begin{Exa}\label{exa.anisi}
Go back to the Examples~\ref{exa1} and~\ref{exa2} which provide examples of gradients which are not isotropically realizable electric fields. However, in the context of Theorem~\ref{thm.v2} we can show that the two gradient fields are realizable electric fields associated with anisotropic conductivities:
\begin{enumerate}
\item Consider the function $u$ defined by \refe{uexa1}., and define the function $v$ by
\beq
v(x):= -\,x_1 + \int_0^{x_2}\chi(t)\,dt,\quad\mbox{for any }x=(x_1,x_2)\in\RR^2.
\eeq
We have
\beq
\nabla v=\chi\left(e_2-e_1\right)+\left(1-\chi\right)\left(-\,e_1\right)\quad\mbox{a.e. in }\RR^2,
\eeq
which combined with \refe{graduexa1}. implies that
\beq
\nabla u\cdot\nabla v=R_\perp\nabla u\cdot\nabla v=1\quad\mbox{a.e. in }\RR^2.
\eeq
Hence, after a simple computation formula \refe{sigraduv}. yields the rank-one laminate (see Section~\ref{ss.lam}) conductivity
\beq
\si=\chi\begin{pmatrix} 2 & 1 \\ 1 & 1 \end{pmatrix}+\left(1-\chi\right){1\over 4}\begin{pmatrix} 1 & 1 \\ 1 & 5 \end{pmatrix}
\quad\mbox{a.e. in }\RR^2.
\eeq
This combined with \refe{graduexa1}. yields
\beq
\si\nabla u=\chi\left(e_1+e_2\right)+\left(1-\chi\right)e_1\quad\mbox{a.e. in }\RR^2,
\eeq
which is divergence free in $\D'(\RR^2)$ since $\left(e_1+e_2-e_1\right)\perp e_1$.
\item Consider the function $u$ defined by \refe{uexa2}., and define the function $v$ by $v(x):=x_2$.
Then, formula \refe{sigraduv}. yields the smooth conductivity
\beq
\si={1\over\left(1+4\pi^2\sin^2(2\pi x_2)\right)^2}
\begin{pmatrix} \left(1+4\pi^2\sin^2(2\pi x_2)\right)^2+4\pi^2\sin^2(2\pi x_2) & -\,2\pi\sin(2\pi x_2) \\ -\,2\pi\sin(2\pi x_2) & 1 \end{pmatrix},
\eeq
This implies that $\si\nabla u=e_1$ which is obviously divergence free in $\RR^2$.
\end{enumerate}
\end{Exa}
\subsection{Global isotropic realizability}
In the previous section we have shown that not all gradients $\nabla u$ satisfying \refe{graduper}. and \refe{gradu­0}. are isotropically realizable when we assume $\sigma$ is periodic. 
In the present section we will prove that the isotropic realizability actually holds in the whole space $\RR^d$ when we relax the periodicity assumption on $\sigma$. To this end consider for a smooth periodic gradient field $\nabla u\in C^1_\sharp(Y)^d$, the following gradient dynamical system
\beq\label{Xtx}
\left\{\ba{rl}
\dis {dX\over dt}(t,x) & =\nabla u\big(X(t,x)\big)
\\ \ecart
X(0,x) & =x,
\ea\right.
\quad\mbox{for }t\in\RR,\ x\in\RR^d,
\eeq
where $t$ will be referred to as the time.
First, we will extend the local rectification result of Theorem~\ref{thm.v1} to the whole space involving a hyperplane. Then, using an alternative approach we will obtain the isotropic realizability in the whole space replacing the hyperplane by an equipotential. Finally, we will give a necessary and sufficient for the isotropic realizability in the torus.
\subsubsection{A first approach}
We have the following result:
\begin{Pro}\label{pro.isoreaH}
Let $u$ be a function in $C^2(\RR^d)$ such that $\nabla u$ satisfies \refe{graduper}. and \refe{gradu­0}..
Also assume that there exists an hyperplane $H:=\{x\in\RR^d:x\cdot\nu=h\}$ such that each trajectory $X(\cdot,x)$ of \refe{Xtx}., for $x\in\RR^d$, intersects $H$ only at one point $z_H(x)=X\big(\tau_H(x),x\big)$ and at a unique time $\tau_H(x)\in\RR$, in such a way that $\nabla u$ is not tangential to $H$ at $z_H(x)$. Then, the gradient $\nabla u$ is an isotropically realizable electric field in $\RR^d$.
\end{Pro}
\begin{Exa}\label{exa2b}
Go back to Example~\ref{exa2} with the function $u$ defined in $\RR^2$ by \refe{uexa2}.. The gradient field $\nabla u$ is smooth and  $Y_2$-periodic.
The solution of the dynamical system \refe{Xtx}. which reads as
\beq
\left\{\ba{ll}
\dis {dX_1\over dt}(t,x)=1, & X_1(0,x)=x_1,
\\ \ecart
\dis {dX_2\over dt}(t,x)=2\pi\sin\big(2\pi X_2(t,x)\big), & X_2(0,x)=x_2,
\ea\right.
\quad\mbox{for }t\in\RR,\ x\in\RR^2,
\eeq
is given explicitly by (see \reff{fig1}.)
\beq\label{Xtx1}
X(t,x)=\left\{\ba{cl}
\left(t+x_1\right)e_1+\left[n+{1\over\pi}\arctan\big(e^{4\pi^2 t}\tan(\pi x_2)\big)\right]e_2 & \mbox{if }x_2\in\left(n-{1\over 2},n+{1\over 2}\right)
\\ \ecart
\left(t+x_1\right)e_1+\left(n+{1\over 2}\right)e_2 & \mbox{if }x_2=n+{1\over 2},
\ea\right.
\eeq
where $n$ is an arbitrary integer.
\begin{figure}
\centering
\vskip -3.cm
\includegraphics[scale=.7]{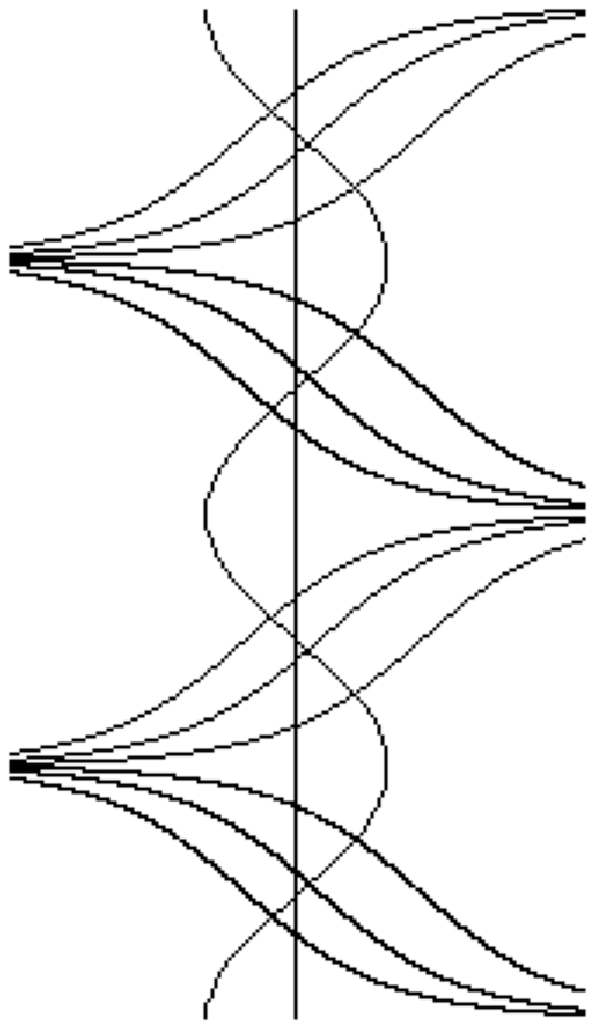}
\vskip -10.5cm
\caption{\it The trajectories crossing the line $\{x_1=0\}$ and the equipotential $\{u=0\}$}
\label{fig1}
\end{figure}
\par
Consider the line $\{x_1=0\}$ as the hyperplane $H$. Then, we have $\tau_H(x)=-\,x_1$.
Moreover, using successively the explicit formula \refe{Xtx1}. and the semigroup property \refe{Xstx}., we get that
\beq
X\big(\!-\!X_1(t,x),X(t,x)\big)=X\big(\!-t-x_1,X(t,x)\big)=X(-\,x_1,x),\quad\mbox{for any }t\in\RR.
\eeq
Hence, the function $v$ defined by $v(x):=X_2(-\,x_1,x)$ satisfies
\beq
v\big(X(t,x)\big)= X_2\big(\!-\!X_1(t,x),X(t,x)\big)=X_2(-\,x_1,x)=v(x),\quad\mbox{for any }t\in\RR.
\eeq
The function $v$ is thus a first integral of system \refe{Xtx}..
It follows that
\beq
{d\over dt}\left[v\big(X(t,x)\big)\right]=0=\nabla v\big(X(t,x)\big)\cdot{dX\over dt}(t,x)=\nabla v\big(X(t,x)\big)\cdot\nabla u\big(X(t,x)\big),
\eeq
which, taking $t=0$, implies that $\nabla u\cdot\nabla v=0$ in $\RR^2$.
Moreover, putting $t=-\,x_1$ in \refe{Xtx1}., we get that for any $n\in\ZZ$,
\beq\label{vn}
v(x)=\left\{\ba{cl}
n+{1\over\pi}\arctan\big(e^{-4\pi^2 x_1}\tan(\pi x_2)\big) & \mbox{if }x_2\in\left(n-{1\over 2},n+{1\over 2}\right)
\\ \ecart
n+{1\over 2} & \mbox{if }x_2=n+{1\over 2}.
\ea\right.
\eeq
Therefore, by \refe{v1}. $\nabla u$ is an isotropically realizable electric field in the whole space $\RR^2$, with the smooth conductivity
\beq
\si:={|\nabla v|\over |\nabla u|}=
\left\{\ba{cl}
\dis {1+\tan^2(\pi x_2)\over e^{4\pi^2 x_1}+e^{-4\pi^2 x_1}\tan^2(\pi x_2)} & \mbox{if }x_2\notin{1\over 2}+\ZZ
\\ \ecart
\dis e^{4\pi^2 x_1} & \mbox{if }x_2\in{1\over 2}+\ZZ.
\ea\right.
\eeq
It may be checked by a direct calculation that $\si\nabla u$ is divergence free in $\RR^2$.
\end{Exa}
\par\bs\noindent
{\bf Proof of Theorem~\ref{pro.isoreaH}.} Let $(\tau_1,\dots,\tau_{d-1})$ be an orthonormal basis of the hyperplane $H$.
Define for each $k\in\{1,\dots,d-1\}$, the function $v_k$ by
\beq\label{vktauk}
v_k(x):=z_H(x)\cdot\tau_k=X\big(\tau_H(x),x\big)\cdot\tau_k,\quad\mbox{for }x\in\RR^d.
\eeq
We shall prove that the functions $v_k$ satisfy the properties of the proof of Theorem~\ref{thm.v1}. $i)$.
\par
First, due the transversality of each trajectory across $H$, we have for any $x\in\RR^d$,
\beq
\left.{\partial\over\partial t}\big(X(t,x)\cdot\nu\big)\right|_{t=\tau_H(x)}=\nabla u\big(z_H(x)\big)\cdot\nu\neq 0.
\eeq
Hence, the implicit functions theorem combined with the $C^1$-regularity of $(t,x)\mapsto X(t,x)$ (see, {\em e.g.}, \cite{Arn}, Theorem~$\rm T'_r$ p.~222) implies that $x\mapsto \tau_H(x)$ defines a function in $C^1(\RR^d)$. Therefore, the functions $v_k$ defined by \refe{vktauk}. belong to $C^1(\RR)$.
\par
Second, since the trajectories satisfy the identity
\beq\label{Xstx}
X\big(s,X(t,x)\big)=X(s+t,x)\quad\forall\,s,t\in\RR,\ \forall\,x\in\RR^d,
\eeq
we get that $X\big(\tau_H(x)-t,X(t,x)\big)=X\big(\tau_H(x),x\big)$, and thus
\beq\label{tHx}
\tau_H\big(X(t,x)\big)=\tau_H(x)-t.
\eeq
It follows that for any $k\in\{1,\dots,d-1\}$,
\beq
v_k\big(X(t,x)\big)=X\big(\tau_H(x)-t,X(t,x)\big)\cdot\tau_k=X\big(\tau_H(x),x)\big)\cdot\tau_k,\quad\mbox{for any }t\in\RR.
\eeq
Therefore, each function $v_k$ is a first integral of the dynamical system \refe{Xtx}..
\par
Third, consider for some $x_0\in\Om$, a vector $(\la_1,\dots,\la_{d-1}\big)\in\RR^{d-1}$ such that
\beq\label{lakgradvk}
\sum_{k=1}^{d-1}\la_k\nabla v_k(x_0)=0,
\eeq
and define the function
\beq\label{v0tau0}
v_0:=\sum_{k=1}^{d-1}\la_k\,v_k=z_H\cdot\tau_0,\quad\mbox{where}\quad\tau_0:=\sum_{k=1}^{d-1}\la_k\,\tau_k.
\eeq
By the chain rule we have
\beq\label{DzH}
\ba{ll}
Dz_H(x) & \dis =\nabla \tau_H(x)\otimes{\partial X\over\partial t}\big(\tau_H(x),x\big)+D_x X\big(\tau_H(x),x\big)
\\ \ecart
& \dis =\nabla \tau_H(x)\otimes\nabla u\big(z_H(x)\big)+D_x X\big(\tau_H(x),x\big).
\ea
\eeq
This combined with the equality (recall that $z_H(x)\in H$)
\beq
0=\left.\nabla\big(z_H(x)\cdot\nu\big)\right|_{x=x_0}=Dz_H(x_0)\,\nu,
\eeq
implies that
\beq\label{tH}
\nabla \tau_H(x_0)={-1\over\nabla u\big(z_H(x_0)\big)\cdot\nu}\,D_x X\big(\tau_H(x_0),x_0\big)\,\nu.
\eeq
Hence, by the equalities \refe{lakgradvk}., \refe{v0tau0}. and again using \refe{DzH}. together with \refe{tH}., we get that
\beq\label{gradv0}
0=\nabla v_0(x_0)=Dz_H(x_0)\,\tau_0
=D_x X\big(\tau_H(x_0),x_0\big)\left(\tau_0-{\nabla u\big(z_H(x_0)\big)\cdot\tau_0\over\nabla u\big(z_H(x_0)\big)\cdot\nu}\,\nu\right).
\eeq
However, by Lemma~\ref{lem.DX} below, the matrix $D_x X\big(\tau_H(x_0),x_0\big)$ is invertible.
This combined with \refe{gradv0}. yields that $\tau_0$ is proportional to $\nu$.
Hence, $\tau_0=0$ and $(\nabla v_1,\dots,\nabla v_{d-1})$ has rank $(d-1)$ everywhere in $\RR^d$.
Therefore, the continuous positive conductivity defined by  \refe{sigradvk}. shows that $\nabla u$ is isotropically realizable in the whole space $\RR^d$.
\cqfd
\begin{Lem}\label{lem.DX}
The derivative $D_xX$ of the dynamical system \refe{Xtx}. is invertible in $\RR\times\RR^d$.
\end{Lem}
\begin{proof}
By the chain rule the matrix field $D_xX$ satisfies the variational equation
\beq
\left\{\ba{rl}
\dis {d\over dt}D_xX(t,x) & =D_xX(t,x)\,\nabla^2 u\big(X(t,x)\big)
\\ \ecart
D_xX(0,x) & =I_d,
\ea\right.
\quad\mbox{for }t\in\RR,\ x\in\RR^d,
\eeq
where $\nabla^2 u$ denotes the Hessian matrix of $u$ and $I_d$ is the unit matrix of $\RR^{d\times d}$.
Moreover, due to the multi-linearity of the determinant $\det\left(D_xX\right)$ satisfies Liouville's formula
\beq
\left\{\ba{rl}
\dis {d\over dt}\det\left(D_xX\right)(t,x) & ={\rm tr }\left[\nabla^2 u\big(X(t,x)\big)\right]\det\left(D_xX\right)(t,x)
\\ \ecart
\det\left(D_xX\right)(0,x) & =1,
\ea\right.
\quad\mbox{for }t\in\RR,\ x\in\RR^d,
\eeq
where tr denotes the trace. It follows that
\beq
\det\left(D_xX\right)(t,x)=\exp\left(\int_0^t {\rm tr }\left[\nabla^2 u\big(X(s,x)\big)\right]ds\right)>0\quad\mbox{for any }(t,x)\in\RR\times\RR^d,
\eeq
which shows that $D_xX(t,x)$ is invertible.
\end{proof}
\begin{Rem}\label{rem.H}
The hyperplane assumption of Theorem~\ref{pro.isoreaH} does not hold in general. Indeed, we have the following heuristic argument:
\par
Let $H$ be a line of $\RR^2$, and let $\Si$ be a smooth curve of $\RR^2$ having an $S$-shape across $H$. Consider a smooth periodic isotropic conductivity $\si$ which is very small in the neighborhood of~$\Si$. Let $u$ be a smooth potential solution of $\div\left(\si\nabla u\right)=0$ in $\RR^2$ satisfying \refe{graduper}., \refe{gradu­0}., and let $v$ be the associated stream function satisfying $\si\nabla u=R_\perp\nabla v$ in $\RR^2$.
The potential $v$ is solution of $\div\left(\si^{-1}\nabla v\right)=0$ in $\RR^2$.
Then, since $\si^{-1}$ is very large in the neighborhood of $\Si$, the curve $\Si$ is close to an equipotential of $v$ and thus close to a current line of $u$.
Therefore, some trajectory of \refe{Xtx}. has an $S$-shape across $H$. This makes impossible the regularity of the time $\tau_H$ which is actually a multi-valued function.
\end{Rem}
\subsubsection{Isotropic realizability in the whole space}
Replacing a hyperplane by an equipotential (see \reff{fig1}. above) we have the more general result:
\begin{Thm}\label{thm.isoreaRd}
Let $u$ be a function in $C^3(\RR^d)$ such that $\nabla u$ satisfies \refe{graduper}. and \refe{gradu­0}..
Then, the gradient field $\nabla u$ is an isotropically realizable electric field in $\RR^d$.
\end{Thm}
\begin{proof}
On the one hand, for a fixed $x\in\RR^d$, define the function $f:\RR\to\RR$ by $f(t):=u\big(X(t,x)\big)$, for $t\in\RR$.
The function $f$ is in $C^3(\RR)$, and
\beq
f'(t)={dX\over dt}(t,x)\cdot\nabla u\big(X(t,x)\big)=\left|\nabla u\big(X(t,x)\big)\right|^2,\quad\forall\,t\in\RR.
\eeq
Since $\nabla u$ is periodic, continuous and does not vanish in $\RR^d$, there exists a constant $m>0$ such that $f'\geq m$ in $\RR$.
It follows that
\beq
{f(t)-f(0)\over t}\geq m,\quad\forall\,t\in\RR\setminus\{0\},
\eeq
which implies that
\beq
\lim_{t\to\infty}f(t)=\infty\quad\mbox{and}\quad\lim_{t\to-\infty}f(t)=-\infty.
\eeq
This combined with the monotonicity and continuity of $f$ thus shows that there exists a unique $\tau(x)\in\RR$ such that
\beq\label{t(x)}
u\big(X(\tau(x),x)\big)=0.
\eeq
\par
On the other hand, similar to the hyperplane case, we have that for any $x\in\RR^d$,
\beq
\left.{\partial\over\partial t}\left[u\big(X(t,x)\big)\right]\right|_{t=\tau(x)}=\left|\nabla u\big(X(\tau(x),x)\big)\right|^2>0.
\eeq
Hence, from the implicit functions theorem combined with the $C^2$-regularity of $(t,x)\mapsto u\big(X(t,x)\big)$ we deduce that $x\mapsto \tau(x)$ is a function in $C^2(\RR^d)$.
\par
Now define the function $w$ in $\RR^d$ by
\beq\label{wuX}
w(x):=\int_0^{\tau(x)}\De u\big(X(s,x)\big)\,ds,\quad\mbox{for }x\in\RR^d,
\eeq
which belongs to $C^1(\RR^d)$ since $u\in C^3(\RR^d)$.
Then, using \refe{Xstx}., \refe{tHx}. and the change of variable $r:=s+t$, we have for any $(t,x)\in\RR\times\RR^d$,
\beq
w\big(X(t,x)\big)=\int_0^{\tau(x)-t}\De u\big(X(s+t,x)\big)\,ds=\int_t^{\tau(x)}\De u\big(X(r,x)\big)\,dr,
\eeq
which implies that
\beq\label{graduw}
{\partial\over\partial t}\left[w\big(X(t,x)\big)\right]=\nabla w\big(X(t,x)\big)\cdot\nabla u\big(X(t,x)\big)=-\De u\big(X(t,x)\big).
\eeq
Finally, define the conductivity $\si$ by
\beq\label{siwu}
\si(x):=e^{w(x)}=\exp\left(\int_0^{\tau(x)}\De u\big(X(s,x)\big)\,ds\right),\quad\mbox{for }x\in\RR^2,
\eeq
which belongs to $C^1(\RR^d)$.
Applying \refe{graduw}. with $t=0$, we obtain that
\beq
\div\left(\si\nabla u\right)=e^w\left(\nabla w\cdot\nabla u+\De u\right)=0\quad\mbox{in }\RR^d,
\eeq
which concludes the proof.
\end{proof}
\begin{Rem}\label{rem.gradu­0}
In the proof of Theorem~\ref{thm.isoreaRd} the condition that $\nabla u$ is non-zero everywhere is essential to obtain both:
\par
- the uniqueness of the time $\tau(x)$ for each trajectory to reach the equipotential $\{u=0\}$,
\par
- the regularity of the function $x\mapsto \tau(x)$.
\end{Rem}
\subsubsection{Isotropic realizability in the torus}
We have the following characterization of the isotropic realizability in the torus:
\begin{Thm}\label{thm.isoreaY}
Let $u$ be a function in $C^3(\RR^d)$ such that $\nabla u$ satisfies \refe{graduper}. and \refe{gradu­0}..
\par\noindent
Then, the gradient field $\nabla u$ is isotropically realizable with a positive conductivity $\si\in L^\infty_\sharp(Y)$, with $\si^{-1}\in L^\infty_\sharp(Y)$, if there exists a constant $C>0$ such that
\beq\label{bdDuX}
\forall\,x\in\RR^d,\quad\left|\,\int_0^{\tau(x)}\De u\big(X(t,x)\big)\,dt\,\right|\leq C,
\eeq
where $X(t,x)$ is defined by \refe{Xtx}. and $\tau(x)$ by \refe{t(x)}..
\par\noindent
Conversely, if $\nabla u$ is isotropically realizable with a positive conductivity $\si\in C^1_\sharp(Y)$, then the boundedness \refe{bdDuX}. holds.
\end{Thm}
\begin{Exa}\label{exa2c}
For the function $u$ of Example~\ref{exa2b} and for $x=(x_1,0)$, we have by \refe{bdDuX}. and \refe{Xtx1}.,
\[
\si_0(x)=\exp\left(4\pi^2\int_0^{\tau(x)}\cos\big(2\pi X_2(s,x)\big)\,ds\right)=\exp\left(4\pi^2\,\tau(x)\right),
\]
and by \refe{t(x)}.,
\[
X_1\big(\tau(x),x\big)=\tau(x)+x_1=\cos\big(2\pi X_2(\tau(x),x)\big)=1.
\]
Therefore, we get that $\si_0(x_1,0)=\exp\big(4\pi^2\,(1-x_1)\big)$, which contradicts the boundedness \refe{bdDuX}..
This is consistent with the negative conclusion of Example~\ref{exa2}.
\end{Exa}
\noindent
{\bf Proof of Theorem~\ref{thm.isoreaY}.}
\par\ss\noindent
{\it Sufficient condition:}
Without loss of generality we may assume that the period is $Y=[0,1]^d$.
Define the function $\si_0$ by
\beq\label{si0}
\si_0(x):=\exp\left(\int_0^{\tau(x)}\De u\big(X(t,x)\big)\,dt\right),\quad\mbox{for }x\in\RR^d,
\eeq
and consider for any integer $n\geq 1$, the conductivity $\sigma_n$ defined by the average over the $(2n+1)^d$ integer vectors of $[-n,n]^d$:
\beq\label{sin}
\sigma_n(x):={1\over(2n+1)^d}\sum_{k\in \ZZ^d\cap[-n,n]^d}\sigma_0(x+k),\quad\mbox{for }x\in\RR^d.
\eeq
\par
On the one hand, by \refe{bdDuX}. $\sigma_n$ is bounded in $L^\infty(\RR^d)$.
Hence, there is a subsequence of $n$, still denoted by $n$, such that $\sigma_n$ converges weakly-$*$ to some function $\sigma$ in $L^\infty(\RR^d)$.
Moreover, we have for any $x\in\RR^d$ and any $k\in\ZZ^d$ (denoting $\dis |k|_\infty:=\max_{1\leq i\leq d}|k_i|$),
\beq
\ba{ll}
\left|\,(2n+1)^d\,\sigma_n(x+k)-(2n+1)^d\,\sigma_n(x)\,\right| &
\dis =\Big|\,\sum_{|j-k|_\infty\leq n}\sigma_0(x+j)-\sum_{|j|_\infty\leq n}\sigma_0(x+j)\,\Big|
\\ \ecart
& \dis \leq\sum_{\ba{c} ^{|j-k|_\infty\leq n} \\[-.15cm] ^{|j|_\infty>n} \ea}\kern -.2cm\sigma_0(x+j)
+\kern -.2cm\sum_{\ba{c} ^{|j-k|_\infty> n} \\[-.15cm] ^{|j|_\infty\leq n} \ea}\kern -.2cm\sigma_0(x+j)
\\
& \leq C\,n^{d-1},
\ea
\eeq
where $C$ is a constant independent of $n$ and $x$.
This implies that $\sigma(\cdot+k)=\si(\cdot)$ a.e. in $\RR^d$, for any $k\in\ZZ^d$. The function $\sigma$ is thus $Y$-periodic and belongs to $L^\infty_\sharp(Y)$. Moreover, since by virtue of \refe{bdDuX}. and \refe{si0}. $\si_0$ is bounded from below by $e^{-C}$, so is $\si_n$ and its limit $\si$. Therefore, $\si^{-1}$ also belongs to $L^\infty_\sharp(Y)$.
\par
On the other hand, by virtue of Theorem~\ref{thm.isoreaRd} the gradient field $\nabla u$ is realizable in $\RR^d$ with the conductivity $\si_0$. This combined with the $Y$-periodicity of $\nabla u$ yields $\div\left(\sigma_n\nabla u\right)=0$ in~$\RR^d$.
Hence, using the weak-$*$ convergence of $\si_n$ in $L^\infty(\RR^d)$ we get that for any $\ph\in C^\infty_c(\RR^d)$, $\nabla u\cdot\nabla\ph\in L^1(\RR^d)$ and 
\beq
0=\lim_{n\to\infty}\int_{\RR^d}\si_n\nabla u\cdot\nabla\ph\,dx=\int_{\RR^d}\si\nabla u\cdot\nabla\ph\,dx.
\eeq
Therefore, we obtain that $\div\left(\sigma\nabla u\right)=0$ in $\D'(\RR^d)$, so that $\nabla u$ is isotropically realizable with the $Y$-periodic bounded conductivity $\sigma$.
\par\bs\noindent
{\it Necessary condition:}
Let $\si$ be a positive function in $C^1_\sharp(Y)$ such that $\div\left(\si\nabla u\right)=0$ in $\RR^d$.
Then, the function $w:=\ln \si$ also belongs to $C^1_\sharp(Y)$, and solves the equation $\nabla w\cdot\nabla u+\De u=0$ in $\RR^d$.
Therefore, using \refe{Xtx}. we obtain that for any $x\in\RR^d$,
\beq
\ba{ll}
\dis \int_0^{\tau(x)}\De u\big(X(t,x)\big)\,dt & \dis =-\int_0^{\tau(x)}\nabla w\big(X(t,x)\big)\cdot\nabla u\big(X(t,x)\big)\,dt
\\ \ecart
& \dis =-\int_0^{\tau(x)}\nabla w\big(X(t,x)\big)\cdot{dX\over dt}(t,x)\,dt
\\ \ecart
& =w\big(X(0,x)\big)-w\big(X(\tau(x),x)\big)=w(x)-w\big(X(\tau(x),x)\big),
\ea
\eeq
which implies \refe{bdDuX}. due to the boundedness of $w$ in $\RR^d$. \cqfd
\begin{Rem}\label{rem.regsi}
If we also assume that $\si_0$ of \refe{bdDuX}. is uniformly continuous in $\RR^d$, then the previous proof combined with Ascoli's theorem implies that the conductivity $\si$ is continuous. Indeed, the sequence $\si_n$ defined by \refe{sin}. is then equi-continuous.
\end{Rem}
\section{The matrix field case}\label{s.mat}
\begin{Def}\label{def.mfield}
Let $\Om$ be an (bounded or not) open set of $\RR^d$, $d\geq 2$, and let $U=(u_1,\dots,u_d)$ be a function in $H^1_{\rm}(\Om)^d$. The matrix-valued field $DU$ is said to be a realizable matrix-valued electric field in $\Om$ if there exists a symmetric positive definite matrix-valued $\si\in L^\infty_{\rm loc}(\Om)^{d\times d}$ such that
\beq
\Div\left(\si DU\right)=0\quad\mbox{in }\D'(\Om).
\eeq
\end{Def}
\subsection{The periodic framework}
\begin{Thm}\label{thm.m1}
Let $Y$ be a closed parallelepiped of $\RR^d$, $d\geq 2$. Consider a function $U\in C^1(\RR^d)^d$ such that
\beq\label{DUper}
DU\mbox{ is $Y$-periodic}\quad\mbox{and}\quad\det\big(\langle DU\rangle\big)\neq 0.
\eeq
\begin{itemize}
\item[$i)$] Assume that
\beq\label{detDU>0}
\det\big(\langle DU\rangle DU\big)>0\quad\mbox{everywhere in }\RR^d.
\eeq
Then, $DU$ is a realizable electric matrix field in $\RR^d$ associated with a continuous conductivity.
\item[$ii)$] Assume that $d=2$, and that $DU$ is a realizable electric matrix field in $\RR^2$, satisfying \refe{DUper}. and associated with a smooth conductivity in $\RR^2$. Then, condition \refe{detDU>0}. holds true.
\item[$iii)$] In dimension $d=3$, there exists a smooth matrix field $DU$ satisfying \refe{DUper}. and associated with a smooth periodic conductivity, such that $\det\left(DU\right)$ takes positive and negative values in $\RR^3$.
\end{itemize}
\end{Thm}
\begin{Rem}\label{rem.m1lreg}
Similarly to Remark~\ref{rem.v2lreg} the assertions $i)$ and $ii)$ of Theorem~\ref{thm.m1} still hold under the less regular assumptions that
\beq\label{DUL2per}
DU\in L^2_\sharp(Y)^{d\times d}\quad\mbox{and}\quad\det\big(\langle DU\rangle DU\big)>0\;\;\mbox{a.e. in }\RR^d.
\eeq
Then, the $Y$-periodic conductivity $\si$ defined by the formula \refe{siCofDU}. below is only defined a.e. in $\RR^d$, and is not necessarily uniformly bounded from below or above in the cell period $Y$. However, $\si DU$ remains divergence free in the sense of distributions on $\RR^d$.
\end{Rem}
\noindent
{\bf Proof of Theorem~\ref{thm.m1}.}
\par\ss\noindent
$i)$ Let $U\in C^1(\RR^d)^d$ be a vector-valued function satisfying \refe{DUper}.. Then, we can define the matrix-valued function $\si$ by
\beq\label{siCofDU}
\si:=\det\big(\langle DU\rangle DU\big)\,(DU^ {-1})^TDU^ {-1}=\det\big(\langle DU\rangle\big)\,{\rm Cof}\left(DU\right)DU^ {-1},
\eeq
where ${\rm Cof}$ denotes the Cofactors matrix. It is clear that $\si$ is a $Y$-periodic continuous symmetric positive definite matrix-valued function. Moreover, Piola's identity (see, {\em e.g.}, \cite{Dac}, Theorem~3.2) implies that
\beq\label{Pio}
\Div\,\big({\rm Cof}\left(DU\right)\big)=0\quad\mbox{in }\D'(\RR^d).
\eeq
Hence, $\si DU$ is Divergence free in~$\RR^d$. Therefore, $DU$ is a realizable electric matrix field associated with the continuous conductivity $\si$.
\par\ms\noindent
$ii)$ Let $DU$ be an electric matrix field satisfying condition \refe{DUper}. and associated with a smooth conductivity in $\RR^2$. By the regularity results for second-order elliptic pde's the function $U$ is smooth in $\RR^2$. Moreover, due to \cite{AlNe} (Theorem~1) we have $\det\big(\langle DU\rangle DU\big)>0$ a.e. in~$\RR^2$. Therefore, as in the proof of Theorem~\ref{thm.v2} we conclude that \refe{detDU>0}. holds. 
\par\ms\noindent
$iii)$ This is an immediate consequence of the counter-example of \cite{BMN} combined with the regularization argument used in the proof of Theorem~\ref{thm.v1} $iii)$. \cqfd
\begin{Rem}
The conductivity $\sigma$ defined by \refe{siCofDU}. can be derived by applying the coordinate change $x'=U^{-1}(x)$ to the homogeneous conductivity $\big|\det\langle DU\rangle\big|\,I_d$.
\end{Rem}
\par
In fact there are many ways to derive a conductivity $\si$ associated with a matrix field $DU$ the determinant of which has a constant sign. Following \cite{Mil2} (Remark p. 155) such a conductivity can be expressed by $\si=JDU^{-1}$, where $J$ is a Divergence free matrix-valued function.
From this perspective we have the following extension of part $i)$ of Theorem~\ref{thm.m1}:
\begin{Pro}\label{pro.siJDU}
Let $U$ be a function in $C^1(\RR^d)^d$ satisfying \refe{DUper}. and \refe{detDU>0}..
Consider a convex function $\ph$ in $C^2(\RR^d)$ whose Hessian matrix $\nabla^2\ph$ is positive definite everywhere in $\RR^d$.
Then, $Du$ is a realizable electric matrix field with the conductivity
\beq\label{siJDU}
\si:=JDU^{-1}\quad\mbox{where}\quad J:=\det\big(\langle DU\rangle\big)\,{\rm Cof}\,\big(D(\nabla\ph\circ U)\big).
\eeq
\end{Pro}
\begin{proof}
On the one hand, the matrix-valued function $J$ of \refe{siJDU}. is clearly Divergence free due to Piola's identity.
On the other hand, we have
\beq
{\rm Cof}\,\big(D(\nabla\ph\circ U)\big)={\rm Cof}\,\big(DU\,\nabla^2\ph\circ U\big)
={\rm Cof}\left(DU\right){\rm Cof}\left(\nabla^2\ph\circ U\right),
\eeq
so that the matrix-valued $\si$ defined by \refe{siJDU}. satisfies
\beq
\si=\det\big(\langle DU\rangle DU\big)\,(DU^ {-1})^T\,{\rm Cof}\left(\nabla^2\ph\circ U\right)DU^ {-1}.
\eeq
Since $\nabla^2\ph$ is symmetric positive definite, so is its Cofactors matrix. Therefore, $\si$ is an admissible continuous conductivity such that $\si DU$ is Divergence free in $\RR^d$.
\end{proof}
\subsection{The laminate case}\label{ss.lam}
\begin{figure}[!ht]
\begin{center}
\begin{tabular}{|p{2.2cm}|p{2.2cm}|p{2.2cm}|p{2.2cm}|p{2.2cm}|p{2.2cm}|}
\hline
& & \centerline{$_{\tex\si^1_{1,2}}$} & & & \centerline{$_{\tex\si^1_{1,2}}$}  \\ \cline{3-3} \cline{6-6}
& & \centerline{$_{\tex\si^2_{1,2}}$} & & & \centerline{$_{\tex\si^2_{1,2}}$}  \\ \cline{3-3} \cline{6-6}
\centerline{$_{\tex\si^1_{1,1}}$} & \centerline{$_{\tex\si^2_{1,1}}$} & \centerline{$_{\tex\si^1_{1,2}}$}
& \centerline{$_{\tex\si^1_{1,1}}$} & \centerline{$_{\tex\si^2_{1,1}}$} & \centerline{$_{\tex\si^1_{1,2}}$}  \\ \cline{3-3} \cline{6-6}
& & \centerline{$_{\tex\si^2_{1,2}}$} & & & \centerline{$_{\tex\si^2_{1,2}}$}  \\ \cline{3-3} \cline{6-6}
& & \centerline{$_{\tex\si^1_{1,2}}$} & & & \centerline{$_{\tex\si^1_{1,2}}$}  \\ \cline{3-3} \cline{6-6}
\hline
\end{tabular}
\end{center}
\caption{\it A rank-two laminate with directions $\xi_{1}=e_1$ and $\xi_{1,2}=e_2$}
\label{fig2}
\end{figure}

\begin{Def}\label{def.lam}
Let~$d,n$ be two positive integers. A rank-$n$ laminate  in~$\mathbb{R}^d$ is a multiscale microstructure defined at~$n$ ordered scales $\ep_n\ll\cdots\ll\ep_1$ depending on a small positive parameter $\ep\to 0$, and in multiple directions in~$\mathbb{R}^d\setminus\{0\}$, by the following process (see~\reff{fig2}.):
\begin{itemize}
\item At the smallest scale~$\ep_n$, there is a set of~$m_n$ rank-one laminates, the $i^{\rm th}$ one of which is composed, for $i=1,\dots,m_n$, of an $\ep_n$-periodic repetition in the direction~$\xi_{i,n}$ of homogeneous layers with constant positive definite conductivity matrices $\si^h_{i,n}$, $h\in I_{i,n}$.
\item At the scale~$\ep_k$, there is a set of~$m_k$ laminates, the $i^{\rm th}$ one of which is composed, for $i=1,\dots,m_k$, of an $\ep_k$-periodic repetition in the direction~$\xi_{i,k}\in\RR^d\setminus\{0\}$ of homogeneous layers and$/$or a selection of the~$m_{k+1}$ laminates which are obtained at stage $(k+1)$ with conductivity matrices $\si^h_{i,j}$, for $j=k+1,\dots,n$, $h\in I_{i,j}$.
\item At the scale $\ep_1$, there is a single laminate ($m_1=1$) which is composed of an $\ep_1$-periodic repetition in the direction~$\xi_{1}\in\RR^d\setminus\{0\}$ of homogeneous layers and$/$or a selection of the~$m_2$ laminates which are obtained at the scale $\ep_2$ with conductivity matrices $\si^h_{i,j}$, for $j=2,\dots,n$, $h\in I_{i,j}$.
\end{itemize}
\end{Def}
The laminate conductivity at stage $k=1,\dots,n$, is denoted by $L^\ep_k(\hat{\si})$, where $\hat{\si}$ is the whole set of the constant laminate conductivities.
\par
Due to the results of \cite{Mil1,Bri} there exists a set $\hat{P}$ of constant matrices in $\RR^{d\times d}$, such that the laminate $P_\ep:=L^\ep_n(\hat{P})$ is a corrector (or a matrix electric field) associated with the conductivity $\si_\ep:=L^\ep_n(\hat{\si})$ in the sense of Murat-Tartar \cite{MuTa}, {\em i.e.}
\beq\label{corPe}
\left\{\ba{rl}
P_\ep\rightharpoonup I_d & \mbox{weakly in }L^2_{\rm loc}(\RR^d)^{d\times d},
\\ \ecart
\Curl\left(P_\ep\right)\to 0 & \mbox{strongly in }H^{-1}_{\rm loc}(\RR^d)^{d\times d\times d},
\\ \ecart
\Div\left(\si_\ep P_\ep\right) & \mbox{is compact in }H^{-1}_{\rm loc}(\RR^d)^d.
\ea\right.
\eeq
The weak limit of $\si_\ep P_\ep$ in $L^2_{\rm loc}(\RR^d)^{d\times d}$ is then the homogenized limit of the laminate.
The three conditions of \refe{corPe}. satisfied by $P_\ep$ extend to the laminate case the three respective conditions
\beq
\left\{\ba{l}
\langle DU\rangle=I_d,
\\ \ecart
\Curl\left(DU\right)=0,
\\ \ecart
\Div\left(\si DU\right)=0,
\ea\right.
\eeq
satisfied by any electric matrix field $DU$ in the periodic case.
\par
The equivalent of Theorem~\ref{thm.m1} for a laminate is the following:
\begin{Thm}\label{thm.m2lam}
Let $n,d$ be two positive integers. Consider a rank-$n$ laminate $L^\ep_n(\hat{P})$ built from a finite set $\hat{P}$ of $\RR^{d\times d}$ (according to Definition~\ref{def.lam}) which satisfies the two first conditions of~\refe{corPe}.. Then, a necessary and sufficient condition for $L^\ep_n(\hat{P})$ to be a realizable laminate electric field, {\em i.e.} to satisfy the third condition of \refe{corPe}. for some rank-$n$ laminate conductivity $L^\ep_n(\hat{\si})$, is that $\det\big(L^\ep_n(\hat{P})\big)>0$ a.e. in $\RR^d$, or equivalently that the determinant of each matrix in $\hat{P}$ is positive.
\end{Thm}
\noindent
{\bf Proof of Theorem~\ref{thm.m2lam}.}
The fact that the determinant positivity condition is necessary was established in \cite{BMN}, Theorem~3.3 (see also \cite{BrNe}, Theorem~2.13, for an alternative approach).
\par
Conversely, consider a rank-$n$ laminate field $P_\ep=L^\ep_n(\hat{P})$ satisfying the two first convergence of \refe{corPe}. and $\det\left(P_\ep\right)>0$ a.e. in $\RR^d$, or equivalently $\det\left(P\right)>0$ for any $P\in\hat{P}$. Similarly to \refe{siCofDU}. consider the rank-$n$ laminate conductivity defined by
\beq\label{siePe}
\si_\ep:=\det\big(P_\ep\big)\,(P_\ep^{-1})^T(P_\ep)^{-1}=L^\ep_n(\hat{\si}),\quad\mbox{where}\quad
\hat{\si}:=\big\{\det\big(P\big)\,(P^ {-1})^TP^ {-1}:P\in\hat{P}\big\}.
\eeq
Then, the third condition of \refe{corPe}. is equivalent to the condition
\beq\label{cofPe}
\Div\,\big({\rm Cof}\left(P_\ep\right)\big)\quad\mbox{is compact in }H^{-1}_{\rm loc}(\RR^d)^d.
\eeq
Contrary to the periodic case ${\rm Cof}\left(P_\ep\right)$ is not divergence free in the sense of distributions. However, following the homogenization procedure for laminates of \cite{Bri}, and using the quasi-affinity of the Cofactors for gradients (see, {\em e.g.}, \cite{Dac}), condition \refe{cofPe}. holds if any matrices $P,Q$ of two neighboring layers in a direction $\xi$ of the laminate satisfy the jump condition for the divergence
\beq\label{cofPQ}
\big({\rm Cof}\left(P\right)-{\rm Cof}\left(Q\right)\big)^T\xi=0.
\eeq
More precisely, at a given scale $\ep_k$ of the laminate the matrix $P$, or $Q$, is:
\begin{itemize}
\item either a matrix in $\hat{P}$,
\item or the average of rank-one laminates obtained at the smallest scales $\ep_{k+1},\dots,\ep_n$.
\end{itemize}
In the first case the matrix $P$ is the constant value of the field in a homogeneous layer of the rank-$n$ laminate. In the second case the average of the Cofactors of the matrices involving in these rank-one laminations is equal to the Cofactors matrix of the average, that is ${\rm Cof}\left(P\right)$, by virtue of the quasi-affinity of the Cofactors applied iteratively to the rank-one connected matrices in each rank-one laminate.
\par
Therefore, it remains to prove equality \refe{cofPQ}. for any matrices $P,Q$ with positive determinant satisfying the condition which controls the jumps in the second convergence of \refe{corPe}., namely
\beq\label{curlPQ}
P-Q=\xi\otimes\eta\quad\mbox{for some }\eta\in\RR^d.
\eeq
By \refe{curlPQ}. and by the multiplicativity of the Cofactors matrix we have
\beq
\ba{ll}
\big({\rm Cof}\left(P\right)-{\rm Cof}\left(Q\right)\big)^T
& ={\rm Cof}\left(Q\right)^T\left[{\rm Cof}\left(I_d+\left(\xi\otimes\eta\right)Q^{-1}\right)^T-I_d\right]
\\ \ecart
& ={\rm Cof}\left(Q\right)^T\left[{\rm Cof}\left(I_d+\xi\otimes\la\right)^T-I_d\right],\quad\mbox{with }\la:=\left(Q^{-1}\right)^T\eta.
\ea
\eeq
Moreover, if $\xi\cdot\la\neq -1$, a simple computation yields 
\beq
{\rm Cof}\left(I_d+\xi\otimes\la\right)^T=\det\left(I_d+\xi\otimes\la\right)\left(I_d+\xi\otimes\la\right)^{-1}
=\left(1+\xi\cdot\la\right)I_d-\xi\otimes\la,
\eeq
which extends to the case $\xi\cdot\la=-1$ by a continuity argument.
Therefore, it follows that
\beq
\big({\rm Cof}\left(P\right)-{\rm Cof}\left(Q\right)\big)^T
={\rm Cof}\left(Q\right)^T\big(\left(\xi\cdot\la\right)I_d-\xi\otimes\la\big),
\eeq
which implies the desired equality \refe{cofPQ}., since $\left(\xi\otimes\la\right)\xi=\left(\xi\cdot\la\right)\xi$. \cqfd
\par\bs\noindent
{\bf Acknowledgement.} The starting point of this work was the matrix field case. The authors wish to thank G.~Allaire who suggested we study the vector field case. GWM is grateful to the INSA de Rennes for
hosting his visits there, and to the National Science Foundation for support through grant DMS-1211359.
%
%

%
%
\end{document}